\documentclass[draft]{amsart}

\usepackage[british]{babel}
\usepackage[usenames,dvipsnames]{color}

\newcommand{\eql}{\kern-1ex &=& \kern-1ex}
\newcommand{\R}{\mathbb{R}}
\newcommand{\C}{\mathbb{C}}
\newcommand{\Ker}{\mathop\mathrm{Ker}\nolimits}
\newcommand{\Ldeg}{L\hbox{-}\mathrm{deg}}
\newcommand{\sign}{\mathop\mathrm{sign}\nolimits}
\newcommand{\w}{\mathbf{w}}
\newcommand{\Pl}{\mathcal P_{\!L}}
\newcommand{\Pt}{\mathcal P_{T}}
\newcommand{\Psil}{\Psi_{\!L}}
\newcommand{\Psit}{\Psi_{T}}
\newcommand{\e}{\varepsilon}
\newcommand{\0}{\mathbf{0}}
\newcommand{\s}{\sigma}
\newcommand{\g}{\gamma}

\renewcommand{\O}{\Omega}

\newcommand{\D}{\mathbf{D}}
\newcommand{\per}{\!\times\!}

\renewcommand{\L}{\mathcal{L}}
\renewcommand{\S}{\mathbf{S}}

\renewcommand{\v}{\mathbf{v}}
\renewcommand{\Im}{\mathop\mathrm{Img}\nolimits}
\renewcommand{\d}{\delta}
\renewcommand{\l}{\lambda}
\renewcommand{\P}{\mathcal P}

\theoremstyle{plain}
\newtheorem{theorem}{Theorem}[section]
\newtheorem{corollary}[theorem]{Corollary}
\newtheorem{lemma}[theorem]{Lemma}

\newtheorem{conjecture}[theorem]{Conjecture}

\newtheorem{remark}[theorem]{Remark}
\theoremstyle{definition}
\newtheorem{notation}[theorem]{Notation}
\newtheorem{definition}[theorem]{Definition}
\newtheorem{example}[theorem]{Example}

\numberwithin{equation}{section}

\begin{document}

\title[Brouwer degree associated to classical eigenvalue problems]{The Brouwer degree associated to classical eigenvalue problems and applications to nonlinear spectral theory}

\author[P.\ Benevieri]{Pierluigi Benevieri}
\author[A.\ Calamai]{Alessandro Calamai}
\author[M.\ Furi]{Massimo Furi}
\author[M.P.\ Pera]{Maria Patrizia Pera}
\thanks{A.\ Calamai is partially supported by G.N.A.M.P.A.\ - INdAM (Italy)}

\thanks
{The first, second and fourth authors are members of the Gruppo Nazionale per l'Analisi Ma\-te\-ma\-ti\-ca, la Probabilit\`a e le loro Applicazioni (GNAMPA) of the Istituto Nazionale di Alta Ma\-te\-ma\-ti\-ca (INdAM)}

\date{\today}

\address{Pierluigi Benevieri -
Instituto de Matem\'atica e Estat\'istica,
Universidade de S\~ao Paulo,
Rua do Mat\~ao 1010,
S\~ao Paulo - SP - Brasil - CEP 05508-090 -
 {\it E-mail address: \tt
pluigi@ime.usp.br}}
\address{Alessandro Calamai -
Dipartimento di Ingegneria Civile, Edile e Architettura,
Universit\`a Politecnica delle Marche,
Via Brecce Bianche,
I-60131 Ancona, Italy -
 {\it E-mail address: \tt
calamai@dipmat.univpm.it}}
\address{Massimo Furi - Dipartimento di Matematica e Informatica
``Ulisse Dini'',
Uni\-ver\-sit\`a degli Studi di Firenze,
Via S.\ Marta 3, I-50139 Florence, Italy -
{\it E-mail address: \tt
massimo.furi@unifi.it}}
\address{Maria Patrizia Pera - Dipartimento di Matematica e Informatica ``Ulisse Dini'',
Universit\`a degli Studi di Firenze,
Via S.\ Marta 3, I-50139 Florence, Italy -
{\it E-mail address: \tt
mpatrizia.pera@unifi.it}}

\begin{abstract}
Thanks to a connection between two completely different topics, the classical eigenvalue problem in a finite dimensional real vector space and the Brouwer degree for maps between oriented differentiable real manifolds, we were able to solve, at least in the finite dimensional context, a conjecture regarding global continuation in nonlinear spectral theory that we formulated in some recent papers.
The infinite dimensional case seems nontrivial, and is still unsolved.
\end{abstract}



\dedicatory{Dedicated to the memory of the outstanding mathematician Andrzej Granas,\\ whose contribution to non-linear analysis has deeply inspired our research}

\maketitle


\section{Introduction}
\label{Introduction}

Consider the \textit{nonlinear eigenvalue problem}
\begin{equation}
\label{problem-intro}
\begin{cases}
\;L\v + s N(\v) = \l \v,\\[.3ex]
\;\v \in \S,
\end{cases}
\end{equation}
where $s$ and $\l$ are real parameters, $L \colon \R^k \to \R^k$ is a linear operator, and $N\colon \S \to \R^k$ is a continuous map defined on the unit sphere of $\R^k$.

One can regard \eqref{problem-intro} as a nonlinear perturbation of the classical eigenvalue problem $L\v = \l\v$, $\v \not= 0$.

By a \emph{solution} of \eqref{problem-intro} we mean any triple $(s,\l,\v) \in \R\per\R\per\S$ which verifies the system.
Solution triples having $s=0$ are called \emph{trivial}.
They correspond to pairs $(\l,\v)$, hereafter called \emph{eigenpoints (of $L$)}, in which $\l \in \R$ is an eigenvalue of $L$ and $\v \in \S$ is one of the associated unit eigenvectors.

Let $\Sigma \subset \R\per\R\per\S$ denote the set of the solutions of \eqref{problem-intro}.
In a recent paper \cite{BeCaFuPe-s3} we proved that \emph{if $\l_* \in \R$ is a simple eigenvalue of $L$ and $\v_*$ is one of the two corresponding unit eigenvectors, then the connected component of $\Sigma$ containing $z_* = (0,\l_*,\v_*)$ is either unbounded or includes a trivial solution $z^* = (0,\l^*,\v^*)$ different from $z_*$.}
Thus, in the second alternative, the eigenvector $\v^*$ must be different from $\v_*$.
Although, according to the statement, one could have $\l_* = \l^*$ and, necessarily, $\v^* = -\v_*$.

Denote by $\mathcal E \subset \R^2$ the projection of $\Sigma$ into the $s\l$-plane.
In \cite{BeCaFuPe-s2} we proved that \emph{if $\l_*$ is a simple eigenvalue of $L$, then the connected component of $\mathcal E$ containing $(0,\l_*)$ is either unbounded or includes a pair $(0,\l^*)$ with $\l^* \not= \l_*$.}

In \cite{BeCaFuPe-s3}, supported by this fact, as well as by the above result regarding $\Sigma$, we formulated the following conjecture that, because of the inadequateness of the topological tools utilized in that article, we were not able to prove or disprove.

\begin{conjecture}
\label{conjecture}
Let $\v_* \in \S$ be a unit eigenvector of $L$ corresponding to a simple eigenvalue $\l_*$.
Then, the connected component of\, $\Sigma$ containing $(0,\l_*,\v_*)$ is either unbounded or includes a triple $(0,\l^*,\v^*)$ with $\l^* \not= \l_*$.
\end{conjecture}

The purpose of this paper is to obtain a Rabinowitz-type continuation result, Theorem \ref{more than ex conjecture}, which implies the positive answer to the above conjecture, and whose proof is now possible thanks to a special link (which, as far as we know, seems new in the literature) between two completely different topics in mathematics, one belonging to algebra and one to topology: namely, classical eigenvalue problems and Brouwer degree theory.

As we shall see, this link depends on the fact that the eigenpoints of $L$ are the zeros of the $C^\infty$ map
\[
\Psil \colon \R\per\S \to \R^k, \quad (\l,\v) \mapsto L\v-\l\v,
\]
whose domain, the cylinder $\R\per\S$, is an orientable, $k$-dimensional, smooth manifold embedded in $\R\per\R^k$.

Thanks to the Brouwer degree, once one of the two possible orientations of the cylinder $\R\per\S$ has been chosen (as we shall see one of them may be considered ``natural''), the map $\Psil$ allows us to assign an integer, $1$ or $-1$, to any eigenpoint $p_* = (\l_*,\v_*)$ of $L$ in which $\l_*$ is a simple eigenvalue.
This integer, which we denote by $\Ldeg(p_*)$ and call \emph{\hbox{$L$-degree} of $p_*$}, can merely be detected by observing the characteristic polynomial of $L$ (Theorem \ref{L-degree isolated eigenpoint}),
and has two crucial properties from which the positive answer to the Conjecture \ref{conjecture} springs.

The first one is a consequence of Theorem \ref{bounded component}:
\emph{If all the trivial solutions contained in a bounded connected component of $\Sigma$ correspond to simple eigenvalues, then the sum of the \hbox{$L$-degrees} of the associated eigenpoints is zero}.

The second one states that the (nonzero) \hbox{$L$-degree} of an eigenpoint $p_* = (\l_*,\v_*)$, corresponding to a simple eigenvalue $\l_*$, is the same as that of its ``twin brother'' $\bar p_* = (\l_*,-\v_*)$; as it should be, since the characteristic polynomial of $L$ ignores which one of the two unit eigenvectors, $\v_*$ or $-\v_*$, one considers.

These two facts imply that the connected component of $\Sigma$ containing the trivial solution $(0,\l_*,\v_*)$, if bounded, must include at least another trivial solution different from $(0,\l_*,-\v_*)$.

\medskip
A pioneer work regarding the persistence of the unit eigenvectors of a perturbed eigenvalue problem is due to R.\ Chiappinelli \cite{Chi}, who examined a problem like \eqref{problem-intro} in the context of a real Hilbert space $H$ (instead of merely $\R^k$), obtaining the so-called \emph{local persistence property} of a unit eigenvector $\v_*$ and the corresponding eigenvalue $\l_*$.
Namely, under the assumptions that $L\colon H \to H$ is a self-adjoint bounded operator, that $\l_* \in \R$ is a simple isolated eigenvalue of $L$, and that $N\colon \S \to H$ is a Lipschitz continuous map defined on the unit sphere of $H$, he proved the existence of two Lipschitz functions, $\e \mapsto \v_\e \in \S$ and $\e \mapsto \l_\e \in \R$, defined on neighborhood $(-\d,\d)$ of $0 \in \R$ and satisfying the following properties:
\[
 \v_0=\v_*,\; \l_0=\l_* \quad \text{and} \quad L\v_\e + \e N(\v_\e) = \l_\e \v_\e,\; \forall\, \e \in (-\d,\d).
\]
Nonlinear eigenvalue problems are related to \emph{nonlinear spectral theory} (see \cite{ADV})
and find applications to differential equations (see e.g.\ the recent survey \cite{Chi2018} and references therein).

Further results regarding the local persistence of eigenvalues, as well as unit eigenvectors, have been obtained in \cite{BeCaFuPe-s1, Chi2017, ChiFuPe1, ChiFuPe2, ChiFuPe3, ChiFuPe4} also in the case in which the eigenvector $\l_*$ is not necessarily simple.
In this framework, a natural question arose: can one prove a sort of ``global persistence'' of eigenvalues and unit eigenvectors?
This was the object of the recent papers \cite{BeCaFuPe-s2, BeCaFuPe-s3, BeCaFuPe-s4}.
As it is of this one; in fact, our main result (Theorem \ref{more than ex conjecture}) ensures the ``global persistence'' of a trivial solution $(0,\l_*,\v_*)$ under the weak assumption that the algebraic multiplicity of the eigenvalue $\l_*$ is odd.
Example \ref{example5} shows that, in Theorem \ref{more than ex conjecture}, the oddness of the algebraic multiplicity of $\l_*$ is crucial.

A possible extension of our main result to the context of real Hilbert spaces seems reasonable, although nontrivial.
A byproduct of the extension would be the positive solution to a question posed in \cite{BeCaFuPe-s4}, which is analogous to Conjecture \ref{conjecture}, but in the infinite dimensional context.
This will be the subject of future investigations.

\section{Notation and preliminaries}
\label{Preliminaries}

Here we introduce some notation, some terminology, and some (more or less) known concepts the we shall need later.

\subsection{Notation and terminology}
Let $E$ and $F$ be two finite dimensional real vector spaces.
By $\L(E,F)$ we shall mean the vector space of the linear operators from $E$ into $F$.
The space $\L(E,E)$ will simply be represented by the symbol $\L(E)$.

The standard inner product of two vectors $\v,\w \in \R^k$ will be denoted by $\langle \v, \w \rangle$.
Thus, the Euclidean norm of an element $\v \in \R^k$ is $\|\v\| =\sqrt{\langle \v, \v \rangle}$.

\smallskip
Sometimes, for the sake of simplicity, we shall use the same symbol for a function and for its restriction to a different domain or codomain, or both of them. In this case, however, the new sets will be evident from the context.
Observe that, according to the formal definition of a function as a triple of sets, domain, codomain and graph, the first set (and, consequently, the third) may be empty.

\begin{notation}
\label{partial map}
Let $X$, $Y$ and $Z$ be metric spaces, $f\colon X \per Y \to Z$ a continuous map, and $x$ a given element of $X$.
By $f_x \colon Y \to Z$ we shall denote the map $y \mapsto f(x,y)$, called the \emph{partial map of $f$ at $x$}.
Moreover, if $D$ is a subset $X \per Y$, by $D_x$ we shall mean the set
$D_x = \{y \in Y: (x,y) \in D\}$, called \emph{$x$-slice of $D$}.
In this case, by $f_x \colon D_x \to Z$ we shall denote the \emph{partial map of $f$ at $x$, relative to $D$}; that is, the restriction to the (possibly empty) slice $D_x$ of the above partial map $f_x\colon Y \to Z$.
If no set $D$ is mentioned, we shall assume $D = X \per Y$, so that $D_x = Y$.
\end{notation}

Hereafter, the boundary and the closure of a subset $A$ of a metric space $X$ will be denoted by $\partial A$ and $\bar A$, respectively.

\begin{definition}
\label{isolated set}
Let $X$ be a locally compact metric space, $C$ a closed subset of $X$, and $K$ a compact subset of $C$. Given an open subset $U$ of $X$, we shall say that \emph{$U$ isolates $K$ (among $C$)} or that \emph{$U$ is an isolating neighborhood of $K$ (among $C$}), if
\begin{itemize}
\item
$\overline U$ is compact;
\item
$U \cap C = K$;
\item
$\partial U \cap C = \emptyset$.
\end{itemize}
In this case, $K$, apart of being compact, is open in $C$.
Notice that any such a subset of $C$ can be isolated by a convenient neighborhood.
In this case we shall say that it is an \emph{isolated subset of $C$}.
\end{definition}

\subsection{Orientation of a finite-dimensional real vector space}
Let $E$ be a finite-dimensional real vector space.
We recall that an \emph{orientation} of $E$ is one of the two equivalence classes of ordered bases of $E$, where
two ordered bases are equivalent if the linear transformation that takes one onto the other has positive determinant.
The space $E$ is \emph{oriented} when one of the two classes, say $\mathcal B$, has been chosen.
In this case an ordered basis is said to be \emph{positive} if it belongs to $\mathcal B$ and \emph{negative} otherwise.
For example, the standard (ordered) basis of $\R^k$ is positive for the \emph{standard orientation} of $\R^k$.

Observe that if a (finite dimensional real) vector space $E$ is a direct sum $E_1 \oplus E_2$ of two oriented subspaces, then, an orientation of $E$, called \emph{compatible (with the splitting)}, may be obtained by the (ordered) union of two positive bases, the first of $E_1$ and the second of $E_2$.
Analogously, if $E_1$ and $E_2$ are oriented spaces, so are (in a natural way) the subspaces $\bar E_1 = E_1 \per \{\0\}$ and $\bar E_2 = \{\0\} \per E_2$ of the product $E_1 \per E_2$. Since $E_1 \per E_2 = \bar E_1 \oplus \bar E_2$, the above argument shows that the orientations of $E_1$ and $E_2$ determine a \emph{compatible orientation} of the product $E_1 \per E_2$.
From this facts one can deduce the following
\begin{remark}
\label{compatible orientation}
Let $E_1$ and $E_2$ be two finite-dimensional real vector spaces. Then, the orientations of two of the three spaces $E_1$, $E_2$ and $E_1 \per E_2$ determine a compatible orientation on the third one. The same holds when the three spaces are $E_1$, $E_2$ and $E$, with $E = E_1 \oplus E_2$.
\end{remark}

Let $L \colon E \to F$ be an isomorphism between oriented finite-dimensional real vector spaces.
The operator $L$ is said to be \emph{orientation preserving [reversing]} if it transforms positive bases of $E$ into positive [negative] bases of $F$.
The \emph{sign} of $L$, $\sign(L)$, is $+1$ if $L$ is orientation preserving, and $-1$ if it is orientation reversing.
Notice that, if $F=E$, then, no matter what is the orientation of $E$, $\sign(L)$ and $\det(L)$ are canonically defined and $\sign(L) =\sign(\det(L))$.
\begin{remark}
\label{product of isomorphism}
Let $L_1\colon E_1 \to F_1$ and $L_2\colon E_2 \to F_2$ be two isomorphisms between oriented, finite dimensional, real vector spaces.
Then, the linear operator
\[
L_1 \per L_2\colon E_1 \per E_2 \to F_1 \per F_2, \quad (\v_1,\v_2) \mapsto (L_1\v_1,L_2\v_2)
\]
is orientation preserving if and only if $L_1$ and $L_2$ are both orientation preserving or both orientation reversing.
More explicitly one has
\[
\sign(L_1 \per L_2) = \sign(L_1) \sign(L_2).
\]
An analogous assertion holds when $F_1 \per F_2$ is replaced by $F = F_1 \oplus F_2$.
\end{remark}

\subsection{Elementary notions on Differentiable Topology}
By a \emph{(differentiable) manifold} we shall mean a smooth (i.e.\ of class $C\sp{\infty}$), boundaryless, real differentiable manifold, embedded in some Euclidean space.

Given a manifold $\mathcal M$ and given $p \in \mathcal M$, the tangent space of $\mathcal M$ at $p$ will be denoted by $T_p(\mathcal M)$.
An \emph{orientation} on $\mathcal M$ is a ``continuous'' map $\omega$, which to any $p \in \mathcal M$ associates an orientation $\omega(p)$ of $T_p(\mathcal M)$. We refer to \cite{Mi} for details regarding the notion of oriented manifolds.

If $f\colon \mathcal M \to \mathcal N$ is a $C\sp{1}$ map between two manifolds and $p \in \mathcal M$, the differential of $f$ at $p$ will be written as $df_p$.
This is a linear operator from $T_p(\mathcal M)$ into $T_{f(p)}(\mathcal N)$.

Remember that, if $f\colon \mathcal M\to \mathcal N$ is a $C\sp{1}$ map, an element $p\in \mathcal M$ is said to be a \emph{regular point} (of $f$) if $df_p$ is surjective.
Non-regular points are called \emph{critical (points)}.
The \emph{critical values} of $f$ are the elements of the target manifold $\mathcal N$ which lie in the image $f(C)$ of the set $C$ of critical points.
Any $q\in \mathcal N$ which is not in $f(C)$ is a \emph{regular value}.
Therefore, in particular, any element of $\mathcal N$ which is not in the image of $f$ is a regular value.

The well-known Sard's Lemma implies that
\emph{the set of regular values of a smooth map $f\colon \mathcal M\to \mathcal N$ between two manifolds is dense in $\mathcal N$.}

\medskip
We recall the following result (see e.g.\ \cite{Mi}).

\begin{theorem}[Regularity of the level set]
\label{Regularity-of-level-set}
Let $f \colon \mathcal M \to \mathcal N$ be a smooth map between two manifolds of dimensions $m$ and $n$, respectively.

If $q \in \mathcal N$ is a regular value for $f$, then $f\sp{-1}(q)$, if nonempty, is a manifold of dimension $m-n$.
Moreover, given $p\in f\sp{-1}(q)$, one has $T_p (f\sp{-1}(q)) = \Ker df_p$.
\end{theorem}

\subsection{The Brouwer degree for maps between oriented manifolds}
We will use a convenient extension, which belongs to the folklore, of the classical Brouwer degree for maps between real finite dimensional oriented manifolds of the same dimension.

We recall, first, some basic notions regarding the classical case, in which the domain manifold is compact or, more generally, the map is ``observed'' on a relatively compact open subset of the domain.
More details can be found, for example, in \cite{Hi,Mi,Ni,OR}.

Let $f\colon \mathcal M \to \mathcal N$ be a continuous map between two oriented manifolds of the same dimension.
Given a value $q \in \mathcal N$ and an open subset $U$ of $\mathcal M$, one says that the triple $(f,U,q)$ is \emph{admissible (for the Brouwer degree)} provided that $U$ isolates $U\cap f^{-1}(q)$ among $f^{-1}(q)$; that is, $U$ is relatively compact in $\mathcal M$ and $q \notin f(\partial U)$.

The Brouwer degree is a (special) function that to any admissible triple $(f,U,q)$ assigns an integer, denoted by $\deg(f,U,q)$ and called the \emph{(Brouwer) degree of $f$ in (the observed set) $U$ with target $q$}.
Roughly speaking, $\deg(f,U,q)$ is an algebraic count of the solutions in $U$ of the equation $f(p) = q$.
In fact, one of the properties of this integer valued function is given by the

\medskip\noindent
\textbf{Computation Formula.}
\label{Computation Formula}
If $(f,U,q)$ is admissible, $f$ is smooth, and $q$ is a regular value for $f$ in $U$, then
\[
\deg(f,U,q)\; = \kern -2 ex \sum_{p \in f^{-1}(q) \cap U}\kern -2 ex \sign(df_p).
\]
This formula is actually the basic definition of the Brouwer degree, and the integer associated to any admissible triple $(g,U,r)$ is defined by
\[
\deg(g,U,r) := \deg(f,U,q),
\]
where $f$ and $q$ satisfy the assumptions of the Computation Formula and are, respectively, ``sufficiently close'' to $g$ and $r$.
It is known that this is a well-posed definition.

For other fundamental properties of the degree that we shall use later, such as Additivity, Excision and Homotopy Invariance, we refer to \cite{Ni}.
Here, we just recall the following useful consequence of the Additivity Property:

\medskip\noindent
\textbf{Existence Property.}
\label{Existence Property}
If $(f,U,q)$ is admissible and $\deg(f,U,q)\neq 0$, then $f^{-1}(q)\cap U$ is nonempty.

\medskip
A special important case regarding the Brouwer degree is when $\mathcal M$ is compact and $\mathcal N$ is connected.
In this case (see, for example, \cite{Mi}) the map $q \in \mathcal N \mapsto \deg(f,\mathcal M,q)$ is constant. Thus, taking into account that the observed set is the whole domain of $f$, one can simply write $\deg(f)$ instead of $\deg(f,\mathcal M,q)$.

A more special, and very interesting case (that we will need later), is when $f$ is a selfmap of a not necessarily oriented manifold $\mathcal M$.
In fact, we have the following
\begin{remark}
\label{selfmap}
If $f$ is a selfmap acting on a compact, connected and orientable manifold $M$, then the integer $\deg(f)$ is well defined and does not depend on the chosen orientation of $\mathcal M$ (assuming that it is the same for domain and codomain).
\end{remark}

\smallskip
It belongs to the folklore the fact that the degree may be extended to triples $(f,V,q)$, from now on called \emph{weakly admissible}, in which $V$ is any open subset of $\mathcal M$ with the unique requirement that $f^{-1}(q)\cap V$ is compact.
In this case one puts
\[
\deg_w(f,V,q) := \deg(f,U,q),
\]
where $U$ is any open subset of $V$ which isolates $f^{-1}(q)\cap V$ among $f^{-1}(q)$.

The Excision Property of the classical Brouwer degree shows that this definition is well posed, even if some properties of the classical degree do not hold anymore in the extended domain of the weakly admissible triples. One of these is the continuity of the map $q \in \mathcal N \setminus f(\partial U) \mapsto \deg(f,U,q)$; another one is the interesting Boundary Dependence Property which is valid for the admissible triples when the target manifold is a vector space.

If we identify $\C$ with $\R^2$, a special weakly admissible triple is $(P,\C,q)$, where $P$ is a non-constant complex polynomial and $q \in \C$.
In this case $\deg_w(P,\C,q)$ does not depend on $q$. Thus, one can simply write $\deg_w(P)$ instead of $\deg_w(P,\C,q)$. One can check that this integer, called \emph{topological degree of $P$}, is the same as the algebraic degree.
Consequently, the Existence Property implies the surjectivity of any non-constant polynomial.

In the context of the weakly admissible triples one has the following folk result that we shall need in the sequel.

\begin{lemma}[Generalized Homotopy Invariance]
\label{Generalized Homotopy Invariance}
Let $\mathcal M$ and $\mathcal N$ be two oriented manifolds of the same dimension.
Given an interval $\mathcal J \subseteq \R$, an open subset $W$ of $\mathcal J \per \mathcal M$, a value $q \in \mathcal N$, and a continuous map $H \colon W \to \mathcal N$, suppose that the subset $H^{-1}(q)$ of $W$ is compact.
Then, the function $s \in \mathcal J \mapsto \deg_w(H_s,W_s,q)$ is constant.
\end{lemma}

The easy proof of Lemma \ref{Generalized Homotopy Invariance} can be performed by showing that the integer valued function $s \mapsto \deg_w(H_s,W_s,q)$ is locally constant.
In fact, this property is a straightforward consequence of the definition of the extended degree and the classical Homotopy Invariance Property.

\subsection{On the sign-jump at a simple eigenvalue of an endomorphism}
\begin{definition}
\label{sign-jump}
Given a real polynomial $P$ and a real root $\l_*$ of $P$, the integer
\[
\lim_{\e\to 0^+}\big(\sign P(\l_*+\e)- \sign P(\l_*-\e)\big) \in \{2,-2,0\}
\]
will be called the \emph{sign-jump of $P$ at $\l_*$}.
\end{definition}

In the next sections we will consider the continuous real function
\[
\P\colon \L(\R^k)\per\R \to \R, \quad (L,\l) \mapsto \det(L-\l I),
\]
where $I$ is the identity in $\R^k$.
Therefore, according to Notation \ref{partial map}, given any operator $L \in \L(\R^k)$, the partial map $\Pl$ of $\P$ at $L$ is the real characteristic polynomial of $L$.

\medskip
Recall that $\l_* \in \R$ is a simple eigenvalue of an endomorphism $L \in \L(\R^k)$ if and only if $\Ker(L-\l_*I) = \R\v_*$ for some $\v_* \notin \Im(L-\l_*I)$.
In this case, putting $T = L-\l_*I$, the restriction $\widehat T\colon \Im T \to \Im T$ of $T$ is an automorphism (\emph{associated to $T$}).
Thus, its determinant is well defined and non-zero.
The following result shows that this determinant is positive if and only if the sign-jump at $\l_*$ of characteristic polynomial $\Pl$ is negative.

\begin{lemma}
\label{sign-jump and associated automorphism}
Let $\l_* \in \R$ be a simple eigenvalue of an endomorphism $L \in \L(\R^k)$.
Then the sign-jump at $\l_*$ of $\Pl$ equals $-2\sign(\widehat T)$, where $\widehat T \in \L(\Im T)$ is the automorphism associated to $T = L-\l_*I$.
\end{lemma}

\begin{proof}
It is enough to prove that the sign-jump at $\l^*= 0$ of the characteristic polynomial $\Pt$ of $T$ equals $-2\sign(\widehat T)$.
Since the eigenvalue $\l^*=0$ of $T$ is simple, one gets the splitting
$
\R^k = \Ker T \oplus \Im T
$
and the matrix representation
\[
T-\l I=
\left(
\begin{array}{cc}
T_{11}-\l I_{11} & T_{12} \\[2ex]
T_{21} & T_{22}-\l I_{22}
\end{array}
\right)
=
\left(
\begin{array}{cc}
-\l I_{11} & 0 \\[2ex]
0 & \widehat T-\l I_{22}
\end{array}
\right).
\]
Therefore, $\Pt(\l) = \det(T - \l I) = -\l\det(\widehat T-\l I_{22})$.

The assertion now follows recalling the definition of sign-jump and observing that $\det(\widehat T -\l I_{22}) \to \det \widehat T$ as $\l \to 0$.
\end{proof}

\section{The eigenvalue problem and the associated Brouwer degree}
\label{Results 1}

Consider the eigenvalue problem
\begin{equation}
\label{classical eigenvalue problem}
\left\{
\begin{aligned}
&L\v = \l \v,\\
&\v \in \S,
\end{aligned}\right.
\end{equation}
where $\l \in \R$, $L \in \L(\R^k)$, and $\S$ is the unit sphere of $\R^k$.

\medskip
An element $(\l,\v)$ of the cylinder $\R\per\S$ will be called an \emph{eigenpoint of $L$} if it satisfies the equation $L\v = \l\v$.
So that, the second element $\v$ is a \emph{unit eigenvector of $L$} corresponding to the (real) \emph{eigenvalue} $\l$.

The set of the eigenpoints of $L$ will be denoted by $\mathcal S$.
Thus, given any $\l \in \R$, the \emph{$\l$-slice} $\mathcal S_\l = \{\v \in \S: (\l,\v) \in \mathcal S\}$ of $\mathcal S$ coincides with $\S \cap \Ker(L-\l I)$.

 Observe that $\mathcal S_\l$ is nonempty if and only if $\l$ is a real eigenvalue of $L$.
 In this case $\mathcal S_\l$ will be called the \emph{eigensphere of $L$ corresponding to $\l$}.
 In fact, it is a sphere whose dimension equals the geometric multiplicity of $\l$ minus one.
In particular, if $\l$ is simple or, more generally, if its geometric multiplicity is one, then $\mathcal S_\l$ has only two elements: the two unit eigenvectors of $L$, one opposite to the other.

If $\l \in \R$ is an eigenvalue of $L$, the nonempty set $\{\l\} \per \mathcal S_\l$ will be called the \emph{eigenset of $L$ corresponding to $\l$} or, briefly, the \emph{$\l$-eigenset of $L$}.
Notice that the set of all the eigenpoints of $L$ is given by
\begin{equation*}
\mathcal S = \bigcup_{\l \in \R} \{\l\} \per \mathcal S_\l.
\end{equation*}

It is convenient to regard $\R\per\S$ as the subset of the space $\R\per\R^k$ satisfying the equation $g(\l,\v) = 1$, where $g \colon \R\per\R^k \to \R$ is defined by $g(\l,\v) = \langle \v,\v \rangle$.
The differential $dg_p \in \L(\R\per\R^k,\R)$ of $g$ at a point $p = (\l,\v)$ is given by $(\dot\l,\dot\v) \mapsto 2\langle \v,\dot\v \rangle$. Therefore, the set of the critical points of $g$ is the $\l$-axis $\v=0$ and, consequently, the number $1$ is a regular value for $g$. This shows that $\R\per\S$ is a smooth manifold of codimension one in $\R\per\R^k$ and, given any $p = (\l,\v) \in g^{-1}(1)$, the tangent space of $\R\per\S$ at $p$ is the kernel of $dg_p$. Namely,
\[
T_{(\l,\v)}(\R\per\S) = \big\{(\dot\l,\dot\v) \in \R\per\R^k: \langle \v,\dot\v\rangle = 0 \big\} = \R\per\v^\perp = (0,\v)^\perp \subset \R\per\R^k.
\]

Observe that the eigenpoints of $L$ are the zeros of the smooth map
\[
\Psil \colon \R\per\S \to \R^k, \quad (\l,\v) \mapsto L\v-\l\v.
\]
Actually, as we shall see, it is convenient to define
\[
\Psi \colon \L(\R^k)\per\R\per\R^k \to \R^k, \quad (T,\l,\v) \mapsto T\v-\l\v,
\]
so that, according to Notation \ref{partial map}, $\Psil$ may be regarded as the partial map of $\Psi$ at $L$, relative to $D = \L(\R^k)\per\R\per\S$.

\medskip
Since $\R\per\S$ is an orientable real differentiable manifold of the same dimension as the oriented space $\R^k$, once we give an orientation to it, it makes sense to consider the Brouwer degree (with target $\0 \in \R^k$), $\deg(\Psil,U,\0)$, of $\Psil$ on any relatively compact open set $U$ such that $\partial U \cap \Psil^{-1}(\0) = \emptyset$.

Since $\R$ has the standard orientation (given by $1$ as a basis), in order to orient $\R\per\S$, it is sufficient to choose one of the two orientations of $\S$.
We prefer the natural one, induced by regarding $\S$ as the boundary of the oriented unit disk $\D$ of $\R^k$.
Namely, given any $\v \in \S$, as a basis of the one-dimensional subspace $\R\v$ of $\R^k$ we take the unit vector $\nu(\v) = \v$ which points outside $\D$.
Consequently, the orientation of the tangent space $T_\v(\S) = \v^\perp$ is the one compatible with the splitting $\R\v \oplus \v^\perp$ of the oriented space $\R^k$ (see Remark \ref{compatible orientation}).

\medskip
In order to simplify some statements, it is convenient to introduce the following notion of \hbox{$L$-degree}, which is well-posed thanks to the Excision Property of the Brouwer degree.

\begin{definition}
\label{definition of L-degree}
Let $K \subset \R\per\S$ be an isolated set of eigenpoints of an operator $L \in \L(\R^k)$.
By the \emph{\hbox{$L$-degree} of $K$} we mean the integer $\Ldeg(K) = \deg(\Psil,U,\0)$, where $U \subset \R\per\S$ is any isolating neighborhood of $K$ among $\Psil^{-1}(\0)$.
In particular, if $p \in \Psil^{-1}(\0)$ is such that the differential $d(\Psil)_p$ is invertible, then $p$ is isolated among $\Psil^{-1}(\0)$ and $\Ldeg(p)=\sign(d(\Psil)_p)$.
\end{definition}

Let $p_* = (\l_*,\v_*)$ be an eigenpoint of $L \in \L(\R^k)$ corresponding to a simple eigenvalue.
In \cite{BeCaFuPe-s3} it was shown that $\Psil\colon \R\per\S \to \R^k$ maps diffeomorphically a neighborhood of $p_*$ in $\R\per\S$ onto a neighborhood of $\0$ in $\R^k$.
This implies that the \hbox{$L$-degree} of $p_*$ is either $1$ or $-1$.
The following result, whose importance is crucial for the remaining part of this paper, shows how $\Ldeg(p_*)$ can be detected by the characteristic polynomial of $L$.

\begin{lemma}
\label{L-degree simple eigenpoint}
Let $\l_*$ be a simple real eigenvalue of an endomorphism $L$ of $\R^k$ and $p_* = (\l_*,\v_*)$ any one of the two corresponding eigenpoints.
Then the \hbox{$L$-degree} of $p_*$ is one-half of the sign-jump at $\l_*$ of the characteristic polynomial of $L$.
\end{lemma}
\begin{proof}
As pointed out in \cite{BeCaFuPe-s3}, and easy to check, the differential
\[
d(\Psil)_{p_*}\colon T_{p_*}(\R\per\S) \to \R^k, \quad (\dot \l, \dot\v) \mapsto (L-\l_*I)\dot\v - \dot\l\v_*
\]
of $\Psil$ at $p_*$ is an isomorphism. Therefore, $\Ldeg(p_*) = \sign(d(\Psil)_{p_*})$.

Hence, by Lemma \ref{sign-jump and associated automorphism}, it is sufficient to show that $d(\Psil)_{p_*}$ and the automorphism $\widehat T \in \L(\Im T)$ associated to $T = L-\l_*I$ have opposite signs.
As $\l_*$ is simple, we have the splitting
\[
\R^k = \Ker T \oplus \Im T = \R\v_* \oplus T(\v_*^\perp).
\]
Since the line $\R\v_*$ is oriented (by the vector $\v_*$), so is, according with the standard orientation of $\R^k$, any of its complement (such as $\v_*^\perp$ and $T(\v_*^\perp)$), as well as the quotient space $\R^k\!\slash\R\v_*$, which is canonically isomorphic to any complement of $\R\v_*$.
This implies that the automorphism $\widehat T$ is orientation preserving if and only if so is the restriction $\widetilde T\colon \v_*^\perp \to T(\v_*^\perp)$ of $T$.

Thus, it is enough to prove that $\widetilde T$ and $d(\Psil)_{p_*}$ have opposite signs.
To this purpose, observe that the domain $T_{p_*}(\R\per\S)$ of the differential $d(\Psil)_{p_*}$ is the product $\R \per T_{v_*}(\S) = \R \per \v_*^\perp$.
Moreover,
\[
d(\Psil)_{p_*}\colon \R \per \v_*^\perp \to \R^k = \R\v_* \oplus T(\v_*^\perp)
\]
is the sum of two maps: the first one, call it $\Lambda$, acts from $\R$ to $\R\v_*$, and is given by $\dot \l \mapsto -\dot\l\v_*$;
the second one is $\widetilde T$.
Therefore, according to Remark \ref{product of isomorphism}, one has $\sign(d(\Psil)_{p_*}) = \sign(\Lambda)\sign(\widetilde T)$.
Since $\Lambda$ is orientation reversing, we finally get $\sign(d(\Psil)_{p_*}) = -\sign(\widetilde T)$.
\end{proof}

As a consequence of Lemma \ref{L-degree simple eigenpoint} we get

\begin{lemma}
\label{L-degree simple eigenset}
Let $\l_*$ be a simple real eigenvalue of $L \in \L(\R^k)$.
Then the \hbox{$L$-degree} of the $\l_*$-eigenset is the sign-jump at $\l_*$ of the characteristic polynomial of $L$.
Consequently, it is either $2$ or $-2$.
\end{lemma}
\begin{proof}
Since $\l_*$ is simple, the eigenset $\{\l_*\}\per\mathcal S_{\l_*}$ is made up of two isolated eigenpoints.
Thus, because of the Additivity Property of the Brouwer degree, its \hbox{$L$-degree} is the sum of the $L$-degrees of the two eigenpoints, and the assertion follows from Lemma~\ref{L-degree simple eigenpoint}.
\end{proof}

The following remark belongs to the folklore, and the easy proof is left to the reader.

\begin{remark}
\label{stability of avoided target}
Let $f\colon X\per K \to Y$ be a continuous map between metric spaces, and let $q \in Y$ be a ``target point''.
Assume that $K$ is compact.
Then, the set of the elements $x \in X$ for which $q$ lies in the image $f_x(K)$ of the partial map $f_x\colon K \to Y$ is closed.
\end{remark}

We recall that, if an operator $T \in \L(\R^k)$ has no eigenpoints in the boundary of a bounded open subset $U$ of $\R\per\S$, then $\deg(\Psit,U,\0)$ is well defined. The next result regards the continuous dependence on $T$ of this integer.

\begin{theorem}
\label{continuous dependence of degree}
Let $U$ be a bounded open subset of the cylinder $\R\per\S$. Then the set $\mathcal U$ of the operators $T \in \L(\R^k)$ without eigenpoints in the boundary of $U$ is open in $\L(\R^k)$.
Moreover, the map $T \in \mathcal U \mapsto \deg(\Psit,U,\0)$ is locally constant.
\end{theorem}
\begin{proof}
Since the boundary $\partial U$ of $U$ is a compact set, the first assertion follows directly from Remark \ref{stability of avoided target}.
Consequently, the map $T \in \mathcal U \mapsto \deg(\Psit,U,\0)$ is well defined.
To prove that it is locally constant, take any $L \in \mathcal U$ and let $\mathcal V \subseteq \mathcal U$ be a convex neighborhood of $L$.
Given $T \in \mathcal V$, it is enough to show that $\deg(\Psit,U,\0) = \deg(\Psil,U,\0)$.
To this purpose, consider the homotopy $H \colon [0,1] \per \overline U \to \R^k$, defined by
\[
(t,(\l,\v)) \mapsto (L + t(T-L))\v - \l\v,
\]
and notice that, because of the convexity of $\mathcal V$, $L + t(T-L) \in \mathcal V$ for all $t \in [0,1]$.
This implies that the partial map $H_t \colon \overline U \to \R^k$, $t \in [0,1]$, never vanishes on $\partial U$.
Thus, because of the Homotopy Invariance Property of the degree, one has
\[
\deg(H_0,U,\0) = \deg(H_1,U,\0),
\]
and the assertion follows from the equalities $H_0 = \Psil$ and $H_1 = \Psit$.
\end{proof}

\begin{theorem}
\label{degree associated to (a,b)}
Given $L \in \L(\R^k)$, let $(a,b)$ be a bounded real interval such that neither $a$ nor $b$ are eigenvalues of $L$.
Then
\[
\deg(\Psil,(a,b)\per \S,\0) = \sign\Pl(b) - \sign\Pl(a),
\]
where $\Pl$ is the characteristic polynomial of $L$.
In particular, if $\l_*$ is any real eigenvalue of $L$, then the \hbox{$L$-degree} of the $\l_*$-eigenset is the sign-jump at $\l_*$ of $\Pl$.
\end{theorem}
\begin{proof}
We need to prove only the first assertion: the last one is a straightforward consequence of this, keeping in mind the definitions of \hbox{$L$-degree} (Definition \ref{definition of L-degree}) and sign-jump (Definition \ref{sign-jump}).

If $(a,b)$ contains only simple eigenvalues (or no eigenvalues), the first assertion follows from the Additivity Property of the Brouwer degree and Lemma \ref{L-degree simple eigenset}.

Otherwise, because of the continuity of the map
\[
\P\colon \L(\R^k)\per\R \to \R, \quad (T,\l) \mapsto \det(T-\l I),
\]
there exists an open neighborhood $\mathcal V$ of $L$ in $\L(\R^k)$ such that
\[
\sign\Pt(b) - \sign\Pt(a) =
\sign\Pl(b) - \sign\Pl(a),
\]
for all $T \in \mathcal V$.

Applying Theorem \ref{continuous dependence of degree} with $U = (a,b)\per \S$, we may also assume that $\mathcal V$ is contained in $\mathcal U$ and that
\[
\deg(\Psit,(a,b)\per \S,\0) = \deg(\Psil, (a,b)\per \S,\0),
\]
for all $T \in \mathcal V$.

The first assertion now follows from that fact that $\mathcal V$ contains operators having only simple eigenvalues.
\end{proof}

The following result extends Lemma \ref{L-degree simple eigenpoint} to the case in which the geometric (but not necessarily the algebraic) multiplicity of an eigenvalue is one.

\begin{theorem}
\label{L-degree isolated eigenpoint}
Let $p_*=(\l_*,\v_*)$ be an isolated eigenpoint of an endomorphism $L$ of $\R^k$.
Then, the \hbox{$L$-degree} of $p_*$ is one-half of the sign-jump at $\l_*$ of the characteristic polynomial of $L$.
\end{theorem}
\begin{proof}
Since $p_*$ is isolated, the geometric multiplicity of $\l_*$ is one, and the eigenpoint $p_*$ has a twin brother, $\bar p_* = (\l_*,-\v_*)$.
Consequently, the $\l_*$-eigenset of $L$ is $\{p_*,\bar p_*\}$, and its \hbox{$L$-degree}, according to Theorem \ref{degree associated to (a,b)}, is the sign-jump at $\l_*$ of the characteristic polynomial $\Pl$ of $L$.
Therefore, because of the Additivity Property of the Brouwer degree, the \hbox{$L$-degree} of $\{p_*,\bar p_*\}$ is the sum of the \hbox{$L$-degrees} of $p_*$ and $\bar p_*$.
Since we do not have any preference, we are incline to believe that the \hbox{$L$-degree} of each one of these points is as in the assertion.
It remains to prove that this is true.

Let $(a,b)$ be an isolating interval of $\l_*$.
Thus, the \hbox{$L$-degree} of $\{p_*,\bar p_*\}$, which is the sign-jump at $\l_*$ of $\Pl$, equals $\sign\Pl(b) - \sign\Pl(a)$.

Now, consider the open subset
$
U = (a,b)\per (\S^+\cup\S^-)
$
of $\R\per\S$, where $\S^+$ is the open hemisphere $\{\v \in \S: \langle \v, \v_*\rangle > 0\}$ and $\S^-$ is the opposite one.
This set $U$ is clearly an isolating neighborhood of the $\l_*$-eigenset $\{p_*,\bar p_*\}$ of $L$.
Thus, its \hbox{$L$-degree}, $\sign\Pl(b) - \sign\Pl(a)$, coincides with $\deg(\Psil,U,\0)$, which, because of the Additivity Property of the degree, equals $\deg(\Psil, U^+,\0) + \deg(\Psil, U^-,\0)$, where $U^+ = (a,b)\per \S^+$ and $U^- = (a,b)\per \S^-$ are isolating neighborhoods of $p_*$ and $\bar p_*$, respectively.
Hence, it remains to show that $\deg(\Psil, U^+,\0) = \deg(\Psil, U^-,\0)$.

As in the proof of Theorem \ref{degree associated to (a,b)}, if $T \in \L(\R^k)$ is sufficiently close to $L$, then
\begin{equation}
\label{equality}
\deg(\Psil,U^+,\0) = \deg(\Psit, U^+,\0)
\quad \text{and} \quad
\deg(\Psil,U^-,\0) = \deg(\Psit, U^-,\0).
\end{equation}
We may assume that $T$ has only simple eigenvalues, so that its eigenset is made up of isolated eigenpoints of $T$ whose \hbox{$T$-degree}, according to Lemma \ref{L-degree simple eigenpoint}, is either $1$ or $-1$, and any eigenpoint (of $T$) in $U^+$ has a twin brother in $U^-$ with the same \hbox{$T$-degree}.
This, because of the Additivity Property of the degree, implies that $\deg(\Psit, U^+,\0) = \deg(\Psit, U^-,\0)$.

The assertion now follows from the equalities \eqref{equality}.
\end{proof}

\section{The perturbed eigenvalue problem and global continuation}
\label{Results 2}

Consider the following perturbed eigenvalue problem:
\begin{equation}
\label{perturbed eigenvalue problem}
\left\{
\begin{aligned}
&L\v + s N(\v)= \l \v,\\
&\v \in \S,
\end{aligned}\right.
\end{equation}
where $s$, $\l$ are real parameters, $L\in \L(\R^k)$, and $N\colon \S \to \R^k$ is a continuous map defined on the unit sphere of $\R^k$.

\medskip
An element $(s,\l,\v) \in \R\per\R\per\S$ is said to be a \emph{solution (triple)} of \eqref{perturbed eigenvalue problem} if it satisfies the equation $L\v + s N(\v)= \l\v$.
With a slight abuse of terminology, the last element $\v$ is said to be a \emph{unit eigenvector} corresponding to the \emph{eigenpair} $(s,\l)$.

We will denote by $\Sigma$ the set of solution triples of \eqref{perturbed eigenvalue problem} and by $\mathcal E$ its projection into the $s\l$-plane, which is the set of the eigenpairs.

\begin{remark}
\label{linear case}
If $N$ is defined on the whole space $\R^k$ and it is linear, then
\[
\mathcal E = \big\{(s,\l) \in \R^2: \det(L + s N -\l I) = 0 \big\}.
\]
\end{remark}

\medskip
Notice that $\Sigma$ is the set of zeroes of the continuous map
\[
\Phi \colon \R\per\R\per\S \to \R^k, \quad (s,\l,\v) \mapsto L\v - \l\v + s N(\v).
\]
Therefore, $\Sigma = \Phi^{-1}(\0)$ is a closed subset of $\R\per\R\per\S$.

The subset $\mathcal E$ of $\R^2$ is closed as well. This is a consequence of Remark \ref{stability of avoided target} with $X = \R^2$, $K = \S$, $Y = \R^k$, $q = \mathbf 0$ and $f = \Phi$.

\medskip
The following result shows, in particular, that, if the space $\R^k$ is odd dimensional, then $\mathcal E$ is unbounded.
Consequently, so is $\Sigma$.

\begin{theorem}
\label{odd dimension}
Let $\R^k$ be odd dimensional.
Then, for any $s \in \R$, there exists at least one $\l \in \R$ such that $(s,\l) \in \mathcal E$.
\end{theorem}
\begin{proof}
Assume, by contradiction, that the assertion is false.
Then, there exists $\check s \in \R$ such that
$
L\v + \check s N(\v) - \l\v \not= \mathbf 0
$
for all $(\l,\v) \in \R\per\S$.

Now, given any $\l \in \R$, consider the continuous map $H_\l\colon \S \to \S$ defined by composing the partial map
\[
\Phi_{(\check s,\l)}\colon \v \mapsto L\v + \check s N(\v) - \l\v \in \R^k\setminus\{\mathbf 0\}
\]
of $\Phi$ at $(\check s,\l)$ with the radial retraction $r\colon \R^k\setminus\{\mathbf 0\} \to \S$.

Since $\S$ is an orientable, compact, connected, real manifold, the Brouwer degree of $H_\l$, $\deg(H_\l)$, is well defined (see Remark \ref{selfmap}) and, because of the Homotopy Invariance Property, does not depend on $\l \in \R$.

Let us show that this contradicts the odd dimensionality of $\R^k$.
To this purpose, we will prove that, if $-a, b \in \R$ are bigger than
$
\max \{\|L\v + \check s N(\v)\|: \v \in \S\},
$
then $H_a$ is homotopic to the identity (of $\S$), therefore $\deg(H_a) = 1$, and $H_b$ is homotopic to the antipodal map, which, because of the even dimensionality of $\S$, has degree $-1$ (see, for example, \cite{Mi}).
This contradiction will prove the assertion.

Indeed, $H_a$ is homotopic to the identity via the composition of the map
\[
(t,\v) \in [0,1]\per \S \mapsto t(L\v+ \check s N(\v)) - a\v \in \R^k\setminus\{\mathbf 0\}
\]
with the radial retraction $r$,
and a similar argument shows that $H_b$ is homotopic to the antipodal map of $\S$.
\end{proof}

Observe that $\Phi_0$, the partial map of $\Phi$ at $s=0$, coincides with the function $\Psil$ defined in Section~\ref{Results 1}.
Moreover, $\Sigma_0$, the slice of $\Sigma$ at $s=0$, is the same as the set $\mathcal S$ of the eigenpoints of $L$.
Notice that the $0$-slice $\mathcal E_0$ of $\mathcal E$ is the set of the real eigenvalues of $L$ and
\begin{equation*}
\label{trivial solutions}
\Sigma_0 = \bigcup_{\l \in \mathcal E_0} \{\l\} \per (\S \cap \Ker(L-\l I)).
\end{equation*}

The solution triples of the type $(0,\l,\v)$ are called \emph{trivial}.
Therefore, the (possibly empty) set of these distinguished elements is $\{0\}\per \Sigma_0$.
Analogously, an eigenpair $(s,\l)$ is said to be \emph{trivial} if $s=0$.
Thus, $\{0\}\per \mathcal E_0$ is the set of the trivial eigenpairs.

\medskip
Once one has a well-defined notion of solution of a given problem (as in the case of \eqref{perturbed eigenvalue problem}) and a distinguished subset of solutions (called \emph{trivial}), one may consider the natural notion of bifurcation point, provided that the set of solutions has a topology.
Therefore, regarding problem \eqref{perturbed eigenvalue problem}, we give the following

\begin{definition}
\label{bifurcation point}
A trivial solution of \eqref{perturbed eigenvalue problem} is a \emph{bifurcation point} (for problem \eqref{perturbed eigenvalue problem}) if it belongs to the closure of the set of nontrivial solutions.
\end{definition}

Identifications between trivial solutions and corresponding aliases are frequent in bifurcation theory.
Thus, if $(0,\l_*,\v_*)$ is a bifurcation point in the sense of Definition \ref{bifurcation point}, we may equivalently call ``bifurcation point'' its \emph{alias} $(\l_*,\v_*)$.
Actually, if it is clear that we are referring to the specific eigensphere $\mathcal S_{\l_*}$, we may simply call ``bifurcation point'' the vector $\v_*$.
Obviously, this is a significative information only when $\mathcal S_{\l_*}$ has positive dimension.
Regarding the case $\dim \mathcal S_{\l_*}>0$, in \cite{ChiFuPe1} a necessary condition as well as some sufficient conditions for a vector $\v_* \in \mathcal S_{\l_*}$ to be a bifurcation point are given.

\medskip

Our next result, Theorem \ref{continuation 1}, regards the existence of bifurcation points for problem \eqref{perturbed eigenvalue problem}.
In its proof we will use the following point-set topology result, which is particularly suited to our purposes and is deduced from general results by C. Kuratowski (see \cite{Ku}, Chapter $5$, Vol.~$2$). We also recommend \cite{Alex} for a helpful paper on connectivity theory.

\begin{lemma}[\cite{FuPe}]
\label{Whyburn}
Let $C$ be a compact subset of a locally compact metric space~$X$.
Assume that every compact subset of $X$ containing $C$ has nonempty boundary.
Then $X \setminus C$ contains a connected set whose closure in $X$ is non-compact and intersects~$C$.
\end{lemma}

\begin{theorem}
\label{continuation 1}
Let $\O$ be an open subset of $\,\R\per\R\per\S$ and let
\[
\O_0 = \{(\l,\v) \in \R\per\S: (0,\l,\v) \in \O\}
\]
denote the slice of $\O$ at $s=0$.
Assume that the degree of $\Psil$ in $\O_0$, $\deg_w(\Psil,\O_0,\0)$, is (defined and) different from zero.
Then $\O$ has a connected set of nontrivial solutions of \eqref{perturbed eigenvalue problem} whose closure in $\O$ is non-compact and contains at least one trivial solutions.
\end{theorem}
\begin{proof}
It is enough to prove that the assumptions of Lemma \ref{Whyburn} are satisfied for
\[
X = \Phi^{-1}(\0) \cap \O \quad \text{and} \quad C = \{0\}\per X_0,
\]
where $X_0$ is the $0$-slice of $X$, so that $C$ is the set of the trivial solutions of \eqref{perturbed eigenvalue problem} which are contained in $\O$.

Notice that $X$ is locally compact, being open in the closed subset $\Phi^{-1}(\0)$ of the finite dimensional manifold $\R\per\R\per\S$.
Moreover, the subset $C$ of $X$ is compact, since its slice $X_0$ coincides with the set $\Psil^{-1}(\0)\cap \O_0$, whose compactness is ensured by the assumption that $\deg_w(\Psil,\O_0,\0)$ is well defined.
Thus, it remains to show that every compact subset of $X$ containing $C$ has nonempty boundary in $X$.

Assume the contrary.
Then, there exists a compact subset $K$ of $X$, containing $C$, whose boundary in $X$ is empty.
This means that $K$ is a relatively open subset of $X$.
That is, there exists an open subset $W$ of $\O$ such that $X \cap W = K$.

Observe that $K$ coincides with $\Phi^{-1}(\0)\cap W$.
Therefore, its compactness ensures that the integer $\deg_w(\Phi_s, W_s,\0)$ is well defined for any $s \in \R$ (recall Notation \ref{partial map}).
Moreover, because of Lemma \ref{Generalized Homotopy Invariance}, this integer does not depend on $s \in \R$.

Since $K$ is compact, there exists $\check s \in \R$ for which the solution set $K_{\check s} \subset W_{\check s}$ is empty. Consequently, because of the Existence Property of the degree, one has $\deg_w(\Phi_{\check s}, W_{\check s},\0)=0$, which implies
$\deg_w(\Phi_s, W_s,\0)=0$ for all $s \in \R$.
This contradicts the assumption $\deg_w(\Psil,\O_0,\0) \not= 0$, since, as we will show, one has $\deg_w(\Phi_0,W_0,\0)=\deg_w(\Psil,\O_0,\0)$.

To see this, observe first that $\Psil$ coincides with the partial map $\Phi_0$ of $\Phi$, so that $\deg_w(\Psil,\O_0,\0) = \deg_w(\Phi_0,\O_0,\0)$.
On the other hand, taking into account that the subset $X_0$ of $\O_0$ is actually contained in $W_0$, we have $\deg_w(\Phi_0,\O_0,\0) = \deg_w(\Phi_0,W_0,\0)$.
Therefore, we get the contradiction $\deg_w(\Phi_0,W_0,\0) \not= 0$, which implies the assertion.
\end{proof}

Recall that the slice $\Sigma_0$ of $\Sigma$ coincides with the set $\mathcal S$ of the eigenpoints of $L$. Therefore, $\{0\}\per\Sigma_0$ is the set of the trivial solutions of \eqref{perturbed eigenvalue problem}.

\begin{corollary}
\label{continuation 2}
Let $U$ be an open subset of $\,\R\per\S$.
Assume that $\deg_w(\Psil,U,\0)$ is (defined and) different from zero.
Then there exists a connected set of nontrivial solutions of \eqref{perturbed eigenvalue problem} whose closure meets $\{0\}\per\Sigma_0$ in $\{0\}\per U$ and is either unbounded or intersects $\{0\}\per\Sigma_0$ outside $\{0\}\per U$.
\end{corollary}
\begin{proof}
Define the open subset $\O$ of $\R\per\R\per\S$ by removing the elements of $\{0\}\per\Sigma_0$ which are not in $\{0\}\per U$. Then, apply Theorem \ref{continuation 1}.
\end{proof}

\smallskip
Observe that $\Sigma_0$ is a finite union of connected components, each of them consisting of an isolated point (when the geometric multiplicity of the corresponding eigenvalue is $1$) or a $\l$-eigenset with $\l$ having geometric multiplicity bigger than $1$. Thus, each one of these components is clopen in $\Sigma_0$. Therefore, if $\mathcal C$ is a connected component of $\Sigma$, then its $0$-slice $\mathcal C_0$ is clopen in $\Sigma_0$ and, consequently, its \hbox{$L$-degree}, $\Ldeg(\mathcal C_0)$, is well defined.

\begin{theorem}
\label{bounded component}
Let $\mathcal C$ be a bounded connected component of the set $\Sigma$ of the solutions of \eqref{perturbed eigenvalue problem}.
Then, the \hbox{$L$-degree} of its $0$-slice, $\Ldeg(\mathcal C_0)$, is zero.
\end{theorem}
\begin{proof}
If $\mathcal C_0$ is empty, the assertion is true.
If this is not the case, assume, by contradiction, that $\Ldeg(\mathcal C_0) \not=0$, and let $U$ be an open subset of $\R\per\S$ which isolates $\mathcal C_0$ among $\Sigma_0$.
Then, by definition of \hbox{$L$-degree}, we get $\deg(\Psil,U,\0) = \Ldeg(\mathcal C_0) \not=0$.
Therefore, recalling that the closure of a connected set is connected, because of Corollary \ref{continuation 2} the compact component $\mathcal C$ contains a connected set of nontrivial solutions of \eqref{perturbed eigenvalue problem} whose closure meets $\{0\}\per\Sigma_0$ outside $\{0\}\per U$.
This contradicts the assumption that $U$ contains $\mathcal C_0$, and the assertion follows.
\end{proof}

In \cite{BeCaFuPe-s3} we conjectured that if $z_*$ is a trivial solution of \eqref{perturbed eigenvalue problem} corresponding to a simple eigenvalue $\l_*$ of $L$, then the connected component of $\Sigma$ containing $z_*$, if bounded, includes a trivial solution $z^*$ corresponding to an eigenvalue $\l^*$ different from $\l_*$.
Theorem \ref{more than ex conjecture} below, which is our main result, provides more than a positive answer to our conjecture: what counts to get the assertion is that the algebraic multiplicity of $\l_*$ is odd, no matter what is its geometric multiplicity.

\begin{theorem}
\label{more than ex conjecture}
Let $\v_*$ be a unit eigenvector of $L$ corresponding to an eigenvalue $\l_*$ with odd algebraic multiplicity.
Then, in the set $\Sigma$ of the solution triples of \eqref{perturbed eigenvalue problem}, the connected component containing $(0,\l_*,\v_*)$ is either unbounded or includes a trivial solution $(0,\l^*,\v^*)$ with $\l^* \not= \l_*$.
\end{theorem}
\begin{proof}
We may suppose that the connected component $\mathcal C$ of $\Sigma$ containing the trivial solution $(0,\l_*,\v_*)$ is bounded.

Thus, we have to show that the slice $\mathcal C_0$ of $\mathcal C$ has at least one eigenpoint which does not belong to the eigenset $\{\l_*\}\per{\mathcal S_{\l_*}}$.
Assuming, by contradiction, $\mathcal C_0 \subseteq \{\l_*\}\per{\mathcal S_{\l_*}}$,
we have two possibilities:
\begin{itemize}
\item
$\mathcal C_0$ coincides with $\{\l_*\}\per{\mathcal S_{\l_*}}$;
\item
$\mathcal C_0$ is strictly contained in $\{\l_*\}\per{\mathcal S_{\l_*}}$ (which is possible only when $\{\l_*\}\per{\mathcal S_{\l_*}}$ is disconnected, and therefore made up of two twin brothers).
\end{itemize}

We claim that none of the alternatives may occur, since, because of the oddness of the algebraic multiplicity of $\l_*$, in both cases the \hbox{$L$-degree} of $\mathcal C_0$ would be nonzero, contradicting Theorem \ref{bounded component}.
In fact, in the first case, because of Theorem \ref{degree associated to (a,b)}, the \hbox{$L$-degree} of $\mathcal C_0$ would be either $2$ or $-2$; while, in the second alternative, $\mathcal C_0$ would coincide with the singleton $\{(\l_*,\v_*)\}$ and, according to Theorem \ref{L-degree isolated eigenpoint}, the \hbox{$L$-degree} would be $1$ or $-1$.

The conclusion is that $\mathcal C_0 \not\subseteq \{\l_*\}\per{\mathcal S_{\l_*}}$, and the assertion is established.
\end{proof}

\section{Examples}
\label{Examples}

In this last section we will present some examples illustrating the assertion of Theorem \ref{more than ex conjecture}.
In particular, Example \ref{example5} will show that, in this theorem, the assumption that the algebraic multiplicity of the eigenvalue $\l_*$ is odd cannot be removed.

For the sake of simplicity, we consider only perturbed eigenvalue problems in which the dimension of the space $\R^k$ is ``very low'' ($1$, $2$ or $3$).
Some of these problems concern linear equations, and in any case it is possible to find explicitly the set of solutions.
We start from the five examples in \cite{BeCaFuPe-s3} - the paper in which Conjecture \ref{conjecture} was formulated - and we reinterpret all of them in the new light of the results of this work.
Furthermore, we add Example \ref{example6}, in which the multiplicity of the eigenvalue $\l_*$ is three.

As in the previous sections, $\Sigma$ and $\mathcal E$ will denote, respectively, the set of solutions and the set of eigenpairs of the given problem.

\medskip
The first example concerns a linear problem in $\R^2$.
Here the operator $L$ has two simple real eigenvalues, and the trivial solutions are four: two for each eigenvalue.
The set of solution triples $\Sigma$ is a smooth curve, diffeomorphic to a circle, which contains all the trivial solutions, and the projection of $\Sigma$ onto $\mathcal E$ is a double covering map.
\begin{example}
\label{example1}
In $\R^2$, consider the problem
\begin{equation}
\label{problem in E1}
\left\{
\begin{array}{rcc}
x_1 - {s} x_2 \eql \l x_1,\\
-x_2 + {s}x_1 \eql \l x_2,\\[0pt]
x_1^2\!+x_2^2 \eql 1.
\end{array}\right.
\end{equation}
Here both $L$ and $N$ are linear: respectively,
\[
L \colon (x_1,x_2) \mapsto (x_1,-x_2) \mbox{ and } N \colon (x_1,x_2) \mapsto (-x_2,x_1).
\]
The operator $L$ has two simple eigenvalues, $\l_* = -1$ and $\l^* = 1$, with two corresponding pairs of antipodal unit eigenvectors:
\[
\pm \v_*= \pm(0,1) \quad \text{and} \quad \pm \v^*= \pm(1,0).
\]
A simple computation (compare also with Remark \ref{linear case}) shows that the set $\mathcal E$ of the eigenpairs of \eqref{problem in E1} is the unit circle $s^2+\l^2 = 1$ in the $s\l$-plane.
So, $\mathcal E$ can be parametrized as $(\sin t, \cos t)$, with $t \in [0,2\pi]$.

Now, to find the set $\Sigma$ of the solution triples, observe that for any fixed $t \in [0,2\pi]$ the kernel of the linear operator
\[
L + (\sin{t})N -(\cos{t})I
\]
is a straight line in the $(x_1,x_2)$-plane which meets the unit sphere $\S$ in the pair of antipodal unit eigenvectors
\[
\pm (\cos (t/2),\sin (t/2)).
\]
Therefore, $\Sigma$ is a bounded subset of $\R\per\R\per\R^2$ which can be expressed parametrically as
\[
\g \colon [0,4\pi] \to \R\per\R\per\R^2, \quad
\g(t) = \Big(\sin t,\cos t,\big(\cos (t/2),\sin (t/2)\big)\Big),
\]
that is, $\Sigma$ is a simple, regular, closed curve
which meets all the four trivial solutions of \eqref{problem in E1} for $t = 0, \pi, 2\pi, 3\pi$.
Incidentally, we observe that the projection of $\Sigma$ onto $\mathcal E$ is a double covering map, and the above parametrization $\g$ of $\Sigma$ is the lifting of the curve
\[
\s\colon [0,4\pi] \to \mathcal E, \quad \s(t) = (\sin t, \cos t),
\]
with the initial condition $\g(0) = \big(0,1,(1,0)\big)$.

Finally, observe that the \hbox{$L$-degree} of the slice $\Sigma_0$ of $\Sigma$ is zero. In fact, $\Sigma_0$ is made up of four isolated eigenpoints and, as shown by Lemma \ref{L-degree simple eigenpoint}, two of them have \hbox{$L$-degree} $1$ and the others have \hbox{$L$-degree} $-1$.
This agrees with Theorem \ref{bounded component}.
Moreover, the bounded and connected set $\Sigma$ satisfies the assertion of Theorem \ref{more than ex conjecture}.
\end{example}

\smallskip
The next example concerns a linear problem in $\R^2$ which is very similar to the previous one; in fact, the unperturbed problem is the same, while the perturbing operator $N$ differs from the one in Example \ref{example1} only for a sign in the first component.
In spite of this, the set of solutions is drastically different: it is composed of four unbounded components, each of them containing only one trivial solution.
\begin{example}
\label{example2}
In $\R^2$, consider the problem
\begin{equation}
\label{problem in E2}
\left\{
\begin{array}{rcc}
x_1 + {s}x_2 \eql \l x_1,\\
-x_2 + {s}x_1 \eql \l x_2,\\[0pt]
x_1^2\!+x_2^2 \eql 1.
\end{array}\right.
\end{equation}
As in the previous example, both $L$ and $N$ are linear: respectively,
\[
L \colon (x_1,x_2) \mapsto (x_1,-x_2) \mbox{ and } N \colon (x_1,x_2) \mapsto (x_2,x_1).
\]
The operator $L$ is the same as in Example \ref{example1}. So, in particular, the four trivial solutions of \eqref{problem in E2} coincide with those of \eqref{problem in E1}; namely,
\[
\big(0,-1,(0,\pm 1)\big)
\quad \text{and} \quad
\big(0,1,(\pm 1,0)\big).
\]
Here, on the other hand,
the set $\mathcal E$ of the eigenpairs of \eqref{problem in E2} is unbounded: it is the hyperbola $\l^2 - s^2 = 1$ in the $s\l$-plane.
The two branches of $\mathcal E$, the lower and the upper one, can be represented parametrically as
\[
({s}_*(t),\l_*(t)) = (\sinh t,-\cosh t)
\quad \text{and} \quad
(s^*(t),\l^*(t)) = (\sinh t,\cosh t),
\]
with $t \in \R$.

To find $\Sigma$, let us first consider the lower branch of $\mathcal E$.
For any fixed $t \in \R$, the kernel of the linear operator
\[
L + {s}_*(t)N -\l_*(t)I
\]
is the straight line containing the pair of opposite (not necessarily unit) vectors
\[
\pm \mathbf w_*(t) = \pm (-\sinh t,1+\cosh t)
\]
and, similarly, concerning the upper branch, the kernel of
\[
L + s^*(t)N -\l^*(t)I
\]
is the straight line containing
\[
\pm \mathbf w^*(t) = \pm (1+\cosh t,\sinh t).
\]
Thus, for example, a parametrization of the component of $\Sigma$ containing the trivial solution $\big(0,1,(1,0)\big)$ is
\[
t \in \R \mapsto \big(\sinh t,\cosh t, \v^*(t)\big),
\]
where $\v^*(t) = \mathbf w^*(t)/\|\mathbf w^*(t)\|$.
In this way we find that $\Sigma$ is made up of four components: the other three can be parametrized in an analogous way.

Obviously, each of these components satisfies the assertion of Theorem \ref{more than ex conjecture}.

In conclusion, the set $\Sigma$ of the solution triples has four unbounded connected components, each of them diffeomorphic to $\R$ and containing one and only one trivial solution, being $1$ or $-1$ the \hbox{$L$-degree} of the corresponding eigenpoint.
This shows that in Theorem \ref{bounded component} the assumption that the component $\mathcal C$ is bounded cannot be removed.
\end{example}

\smallskip
In the following example the space is again $\R^2$ and $N$ is nonlinear (actually, is constant).
Here, as in Examples \ref{example1} and \ref{example2}, $L$ has two different real eigenvalues, so that the trivial solutions are four.
The set $\Sigma$ is the union of a topological circle $\mathcal C$ and two straight lines.
In spite of the fact that $\mathcal C$ contains only two trivial solutions with the \emph{same} eigenvalue, $\Sigma$ connects all the four trivial solutions, compatibly with Theorem~\ref{more than ex conjecture}.
\begin{example}
\label{example3}
In $\R^2$, consider the problem
\begin{equation}
\label{problem in E3}
\left\{
\begin{array}{lcc}
x_1 + s\eql \l x_1,\\
2x_2 \eql \l x_2,\\[0pt]
x_1^2\!+x_2^2 \eql 1.
\end{array}\right.
\end{equation}
Here we have
\[
L \colon (x_1,x_2) \mapsto (x_1,2x_2) \mbox{ and } N \colon (x_1,x_2) \mapsto (1,0).
\]
The operator $L$ has two simple eigenvalues, $\l_* = 1$ and $\l^* = 2$, with corresponding unit eigenvectors, $\pm \v_*= \pm(1,0)$ and $\pm \v^*= \pm(0,1)$.
So, the four trivial solutions of \eqref{problem in E3} are
\[
\big(0,1,(\pm 1,0)\big)
\quad \text{and} \quad
\big(0,2,(0,\pm 1)\big).
\]

Regarding $\Sigma$, notice that any solution $\big(s,\l,(x_1,x_2)\big)$ of \eqref{problem in E3} must verify either $x_2=0$ (and, consequently, $x_1=\pm1$) or $\l=2$ (and, therefore, $x_1={s}$ with $|{s}| \le 1$).
In the first case, with $x_2=0$, we get two straight lines in $\R\per\R\per\R^2$:
\[
\ell_1=\big\{\big(s,1-s,(-1,0)\big): s\in \R \big\}
\quad
\text{and}
\quad
\ell_2=\big\{\big(s,1+s,(1,0)\big): s\in \R \big\},
\]
lying in the two different planes of equations $(x_1,x_2)=(-1,0)$ and $(x_1,x_2)=(1,0)$, respectively.
In the second case, we observe that the set of solutions having $\l=2$ can be represented as follows:
\begin{equation*}
\label{smcurve}
\mathcal C=\big\{\big(\sin t,2,(\sin t,\cos t)\big): t \in [0,2\pi] \big\}.
\end{equation*}

The set $\mathcal C$ is diffeomorphic to a circle and contains two - and only two - of the four trivial solutions, but both with the same eigenvalue $\l^*=2$.
If $\mathcal C$ was a connected component of $\Sigma$, this would be in contrast with Theorem \ref{more than ex conjecture}. However, $\Sigma$, which is the union of $\ell_1$, $\ell_2$ and $\mathcal C$, is connected, since
\[
\ell_1 \cap\mathcal C=\big\{\big(-1,2,(-1,0)\big)\big\} \quad
\text{and} \quad \ell_2 \cap\mathcal C=\big\{\big(1,2,(1,0)\big)\big\}.
\]
Therefore, the connected component of $\Sigma$ containing any of the four trivial solutions is $\Sigma$ itself, which satisfies the assertion of Theorem~\ref{more than ex conjecture}.
\end{example}

\smallskip
The simplest example that one can conceive is when the space is one-dimensional.
In this case the unit sphere $\S$ is $\{-1,1\}$ and, whatever is $N$, the set $\Sigma$ consists of two unbounded connected components.

\begin{example}
\label{example4}
Let $\l_* \in \R$ be given and, in $\R$, consider the problem
\begin{equation}
\label{problem in E4}
\left\{
\begin{array}{rcc}
\l_* x + {s}N(x) \eql \l x,\\[0pt]
x \eql \pm 1,
\end{array}\right.
\end{equation}
where $N\colon \{-1,1\} \to \R$ is arbitrary.
Given any $s\in \R$, one has two solutions:
\[
(s, \l_* + sN(1),1) \quad \text{and} \quad (s, \l_* - {s}N(-1),-1).
\]
Thus $\Sigma$ is composed of two straight lines:
\[
\big\{\big(s, \l_* + {s}N(1),1\big): s\in \R \big\}
\quad
\text{and}
\quad
\big\{\big(s, \l_* - {s}N(-1),-1\big): s\in \R \big\},
\]
contained in two different planes of $\R^3$.
This is compatible with Theorem~\ref{more than ex conjecture}.
\end{example}

\smallskip
The following example in $\R^3$ shows that, in Theorem \ref{more than ex conjecture}, the assumption that the algebraic multiplicity of the eigenvalue $\l_*$ is odd cannot be omitted.
\begin{example}
\label{example5}
In $\R^3$, consider the problem
\begin{equation}
\label{problem in E5}
\left\{
\begin{array}{rcc}
{s}x_2 \eql \l x_1,\\
2x_1 - {s}x_1 \eql \l x_2,\\
2x_3 + {s}x_1 \eql \l x_3,\\
x_1^2 + x_2^2 + x_3^2 \eql 1.
\end{array}\right.
\end{equation}
Here both $L$ and $N$ are linear: respectively,
\[
L \colon (x_1,x_2,x_3) \mapsto (0,2x_1,2x_3) \mbox{ and } N \colon (x_1,x_2,x_3) \mapsto (x_2,-x_1,x_1).
\]
The operator $L$ has two real eigenvalues: $\l_* = 0$, with geometric multiplicity $1$ and algebraic multiplicity $2$; and $\l^*= 2$, which is simple.
So, \eqref{problem in E5} has four trivial solutions given by
\[
\big(0,0,(0,\pm 1,0)\big)
\quad \text{and} \quad
\big(0,2,(0,0,\pm 1)\big).
\]
It is not difficult to show (compare with Remark \ref{linear case}) that the set $\mathcal E$ of the eigenpairs has two connected components: the circle $({s}-1)^2+\l^2 = 1$ and the straight line $\l=2$.
In particular, notice that the disconnected set $\mathcal E$ is unbounded, as it should be, according to Theorem \ref{odd dimension}.
Consequently, $\Sigma$ is as well unbounded and disconnected.

Now, let us investigate the solution set $\Sigma$.
To this purpose, observe that the bounded connected component of $\mathcal E$ can be parametrized as
\[
(s(t),\l(t)) = (1-\cos{t},\sin{t}), \quad \text{with} \quad t \in [0,2\pi].
\]
For any fixed $t \in [0,2\pi]$, the kernel of the linear operator
\[
L + {s}(t)N -\l(t)I
\]
is spanned by the (nonzero) vector
\[
\mathbf w(t) = \big(\sin(t/2),\cos(t/2),c(t)\big),
\]
where $c(t)$ is defined by $2c(t) + {s}(t)\sin(t/2) = \l(t)c(t)$, in order to satisfy the third equation of \eqref{problem in E5}.
This implies that the connected component $\mathcal C$ of $\Sigma$ containing the trivial solution $z_*= \big(0,0,(0,1,0)\big)$ can be parametrized as follows:
\[
\s \colon [0,4\pi] \to \R\per\R\per S^2, \quad
\s(t) = \big(1-\cos{t},\sin{t},\v(t)\big),
\]
where $\v(t) = \mathbf w(t)/\|\mathbf w(t)\|$.
That is, $\mathcal C$ is diffeomorphic to $S^1$ - in particular, $\mathcal C$ is bounded - and contains both the trivial solutions corresponding to the eigenvalue $\l_*=0$: $\big(0,0,(0,1,0)\big)$ for $t=0$ (or, equivalently, for $t=4\pi$) and $\big(0,0,(0,-1,0)\big)$ for $t=2\pi$.
Notice that the other two trivial solutions, the ones corresponding to $\l^* = 2$, do not belong to $\mathcal C$.
Therefore, the slice $\mathcal C_0$ at $s=0$ is the $\l_*$-eigenset of $L$.

Since the algebraic multiplicity of $\l_*$ is even, the sign-jump at $\l_*$ of the characteristic polynomial of $L$ is zero.
Consequently, due to Theorem \ref{degree associated to (a,b)}, the \hbox{$L$-degree} of the $\l_*$-eigenset is zero, which shows that the bounded component $\mathcal C$ satisfies the statement of Theorem \ref{bounded component}.

However, $\mathcal C$ does \emph{not} satisfy the assertion of Theorem \ref{more than ex conjecture}. Hence, in this result, the assumption that the algebraic multiplicity of $\l_*$ is \emph{odd} cannot be removed.

The other two components of $\Sigma$, the ones containing the trivial solutions corresponding to the simple eigenvalue $\l^*=2$, do satisfy the assertion of Theorem~\ref{more than ex conjecture}.
In fact, they are the straight lines
\[
\big\{\big(s,2,(0,0,1)\big): s\in \R\big\}
\quad \text{and} \quad
\big\{\big(s,2,(0,0,-1)\big): s\in \R\big\}.
\]

Finally, we observe that the projection of $\mathcal C$ onto the circle $(s-1)^2+\l^2 = 1$ is a double covering map, and the above parametrization $\s$ of $\mathcal C$ is the lifting of the curve $t \in [0,4\pi] \mapsto (1-\cos{t},\sin{t})$ with initial condition $\s(0) = z_* = \big(0,0,(0,1,0)\big)$.
\end{example}

We close with an example in $\R^3$ in which the linear operator $L$ has an eigenvalue $\l_*$ with algebraic multiplicity three and geometric multiplicity two. The $\l_*$-eigensphere is one-dimensional and contains two bifurcation points.

\begin{example}
\label{example6}
In $\R^3$, consider the problem
\begin{equation}
\label{problem in E6}
\left\{
\begin{array}{rcc}
x_1+x_3+{s}x_1 \eql \l x_1,\\
x_2 - {s}x_3 \eql \l x_2,\\
x_3 + {s}x_2 \eql \l x_3,\\
x_1^2 + x_2^2 + x_3^2 \eql 1.
\end{array}\right.
\end{equation}
Here one has
\[
L\colon (x_1,x_2,x_3) \mapsto (x_1+x_3,x_2,x_3) \quad \text{and} \quad
N\colon (x_1,x_2,x_3) \mapsto (x_1,-x_3,x_2).
\]
The operator $L$ has only one eigenvalue: $\l_* = 1$, with algebraic multiplicity $3$ and geometric multiplicity $2$.
Applying Remark \ref{linear case} we get
\[
\mathcal E = \big\{(s,\l) \in \R^2:
(1-\l+s)\big((1-\l)^2+s^2\big)=0\big\}.
\]
Therefore, $\mathcal E$ is the straight line in the $s\l$-plane of equation $\l = 1+s$.
This line includes the unique trivial eigenpair $(0,\l_*) = (0,1)$ corresponding to the eigensphere
\[
\mathcal S_{\l_*} = \big\{(x_1,x_2,x_3) \in \R^3: x_1^2+x_2^2 = 1,\; x_3=0\big\}.
\]
One can check that $\Sigma$ contains the straight lines
\[
\Sigma_- = \big\{\big(s,1+s,(-1,0,0)\big)\in \R\per\R\per\R^3: s \in \R\big\}
\]
and
\[
\Sigma_+ = \big\{\big(s,1+s,(1,0,0)\big)\in \R\per\R\per\R^3: s \in \R\big\}.
\]
Consequently, $\mathcal S_{\l_*}$, which is the unit circle in the plane $x_3 = 0$, contains two (aliases of) bifurcation points:
$(-1,0,0)$ and $(1,0,0)$.

Notice that $\Sigma$ is the union of three connected sets:
$\Sigma_-$, $\Sigma_+$, and the circle
\[
\{0\}\per\{1\}\per\mathcal S_{\l_*} \subset \R\per\R\per\R^3,
\]
which intersects both $\Sigma_-$ and $\Sigma_+$.
Thus, $\Sigma$ is connected and, being unbounded, the assertion of Theorem \ref{more than ex conjecture} is satisfied.
\end{example}



\end{document}